\documentclass[11pt]{amsart}
\usepackage{amsmath,amsfonts,amssymb,overpic,enumerate}
\usepackage{graphicx}
\usepackage[a4paper]{geometry}
\usepackage[small,margin=0pt]{caption}

\theoremstyle{plain}
\newtheorem{thm}{Theorem}
\newtheorem{lem}[thm]{Lemma}
\newtheorem{cor}[thm]{Corollary}
\newtheorem{prop}[thm]{Proposition}

\theoremstyle{Definition}

\newtheorem{ex}[thm]{Example}
\newtheorem{assumption}{Assumption}

\theoremstyle{remark}

\DeclareMathOperator*{\argmin}{arg\,min}

\def\dist{\operatorname{dist}}
\def\grad{\operatorname{grad}}
\def \<{\langle}
\def \>{\rangle}
\def\N{{\mathbb N}}
\def\Z{{\mathbb Z}}

\def\Exende{\def\qedsymbol{$\diamondsuit$}\qed}
\def\CH{Car\-tan-\hskip0pt Ha\-da\-mard}

\begin{document}

\title{Convergence analysis of subdivision processes on the sphere}

\author{Svenja H\"uning, Johannes Wallner}
\address{\newline
Institut
f.\ Geometrie, TU Graz. Kopernikusgasse 24, 8010 Graz, Austria.
\newline
email $\{$huening,j.wallner$\}$@tugraz.at}

\maketitle

\begin{abstract}
{This paper provides a strategy to analyse the convergence of nonlinear analogues of linear subdivision processes on the sphere. 
In contrast to previous work, we study the Riemannian analogue of a linear scheme on a Riemannian manifold with positive sectional curvature. Our result can be applied to all general subdivision schemes without any sign restriction on the mask.}
{refinement algorithm; approximation theory; differential geometry.}
\end{abstract}

\section{Introduction}

Subdivision schemes are iterative refinement algorithms used to produce smooth curves and surfaces. For data lying in linear spaces those refinement rules are well-studied and find applications in various areas ranging from approximation theory to computer graphics. We recommend \cite{dyn2, cavaretta} and \cite{peters} for an introduction and good overview on this topic.

This paper contributes new results to the convergence analysis of subdivision schemes applied to data lying in nonlinear spaces, such as Lie groups, symmetric spaces or Riemannian manifolds. 
Several different methods to transfer linear schemes to nonlinear geometries have been studied (see \cite{grohs} for an overview). 
We mention the \emph{log-exp-analogue} which uses the exponential map and the linear structure of the tangent space by \cite{donoho} and \cite{rahman} as well as the \emph{projection analogue} which can be applied to surfaces embedded into an Euclidean space, see \cite{xie2} and \cite{grohs5}. Generally, such methods are only locally well-defined.
The \emph{Riemmanian analogue} which is obtained from a linear scheme by replacing affine averages by weighted geodesic averages can be made globally well-defined on complete Riemannian manifolds with nonpositive sectional curvature \cite{huening2019, wallner, ebner1,ebner2}. The well-definedness is based on the existence and uniqueness of the so-called \emph{Riemannian center of mass} which has been studied in various contexts \cite{karcher, sander, dyer-arxiv, dyer2016, pennec}. 
 
Having transferred the linear scheme to nonlinear geometries questions of their properties arise naturally. 
While the smoothness of resulting limit curves is already thoroughly investigated, see e.g.\  \cite{wallner3, grohs3, wallnerdyn,xie2,grohs, caroline1}, the situation is different for the convergence analysis of nonlinear analogues of linear schemes. 
Results in literature based on so-called \emph{proximity conditions} are limited to `dense enough' input data. Convergence results for all input data are given for univariate interpolatory schemes \cite{wallner2}, multivariate schemes with nonnegative mask \cite{ebner1, ebner2} or schemes which are obtained from linear ones by choosing binary geodesic averages instead of linear binary averages \cite{dyn4, dyn3}. In \cite{huening2019} we proved that the Riemannian analogue of a univariate linear scheme on complete Riemannian manifolds with nonpositive sectional curvature converges if the linear scheme converges uniformly. This result was previously only known for schemes with nonnegative mask \cite{wallner,ebner1, ebner2}.

In this paper, we extend our work to a convergence results for refinement algorithms on the sphere. It turns out that even for this rather elementary manifold the situation is appreciably different from the case of nonpositive curvature.
   
The paper is organised as follows. We start by repeating some basic facts about linear subdivision and introduce our notation. Next, we discuss the Riemannian analogue of a linear scheme on positively curved spaces. In particular, we discuss well-definedness on the sphere. Afterwards, we introduce our strategy to prove convergence. To make this rather technical part easier to understand we illustrate our computations in terms of a major example (the cubic Lane-Riesenfeld scheme). The last part contains the convergence analysis of several well-known subdivision schemes.

\section{Linear subdivision and its Riemannian analogue} \label{section_sphere}
We start by recalling basic concepts. 
Let $(x_i)_{i\in \mathbb{Z}}$ be a sequence of points in a linear space. We refer to this sequence as our \emph{input data}. A linear, binary \emph{subdivision rule} $S$ maps the input data to a new sequence $(Sx_i)_{i\in \mathbb{Z}}$ where
\begin{align*}
Sx_i=\sum_{j\in \mathbb{Z}}a_{i-2j}x_j.
\end{align*}
The sequence of coefficients $a=(a_i)_{i\in \mathbb{Z}}$ is called the \emph{mask} of the scheme. Throughout this paper we assume that the mask is finitely supported, i.e., $a_i\neq 0$ for only finitely many $i\in \mathbb{Z}$. 
A \emph{subdivision scheme} is the repeated application $S$, $S^2$, $S^3$, $\dots$ of the subdivision rule. \emph{Affine invariance}, resp. \emph{translation invariance}, of a linear subdivision scheme is expressed by
\begin{align} \label{affine invariance}
\sum_{j\in \mathbb{Z}} a_{2j}=\sum_{j\in \mathbb{Z}} a_{2j+1}=1.
\end{align}
It is a necessary condition for its convergence \cite{cavaretta, dyn2}.
From now on, we assume that all considered schemes are affine invariant.

We call a linear subdivision scheme $S$ convergent, if there exist piecewise linear functions $f_k$ with $f_k\left(\frac{i}{2^k}\right)=\left(Sx\right)_i$
which converge, uniformly on compact sets, to a continuous limit.\\

Denote by $M$ a manifold with Riemannian metric. The \emph{distance} of two points $p$, $q \in M$ is bounded by the length over all $\gamma:[0,1]\rightarrow M$ with $\gamma(0)=p$ and $\gamma(1)=q$, i.e.\
\begin{align*}
\dist\left(p,q\right):=\inf_{\gamma}\int_{0}^1 \vert \dot{\gamma}(t) \vert~dt.
\end{align*}

To adapt a subdivision scheme to data lying in nonlinear geometries we replace affine averages by geodesic averages.
Therefore, we observe that the point $Sx_i=\sum_{j\in \Z}a_{i-2j}x_j$ can be equivalently described as 
\begin{align*}
\argmin_{x} \sum_{j\in \mathbb{Z}}a_{i-2j}\hspace{0.05cm}\vert x-x_j\vert^2.
\end{align*}
A natural extension of $S$ to nonlinear data is defined by replacing the Euclidean distance by the Riemannian distance. We define 
\begin{align} \label{min:fct}
Tx_{i}:=\argmin_{x}\sum_{j\in \mathbb{Z}}a_{i-2j}\dist\left(x,x_j\right)^2, \hspace{0.2cm} i\in \Z,
\end{align}
and call $T$ the \emph{Riemannian analogue} of the linear subdivision rule $S$, see also \cite{grohs}. The minimiser of $\sum_{j\in \mathbb{Z}}a_{i-2j}\dist\left(x,x_j\right)^2$ is the \emph{Riemannian center of mass} of points $x_j$ with respect to weights $a_{i-2j}$.

Locally, the minimiser (\ref{min:fct}) always exists but globally the question is more difficult. 
On Cartan-Hadamard manifolds, i.e.\ simply connected, complete Riemannian manifolds with nonpositive sectional curvature, a unique, global minimiser always exists \cite{hardering, sander,karcher}. 
 
On positively curved spaces, however, the situation is different. 
We regard this problem in more detail in the next section. Before, we introduce the contractivity condition and the displacement-safe condition we need in our analysis.

We say that $T$ satisfies a \emph{contractivity condition} with \emph{contractivity factor} $\mu \in (0,1)$, if
\begin{align} \label{contractivity}
	\dist (T^{k}x_{i+1},T^{k}x_{i})
	\leqslant \mu^{k}\cdot 
	\sup\limits_{\ell} \dist (x_{\ell},x_{\ell+1}), \hspace{0.2cm} i\in \Z, ~k\in \N .
\end{align}  
The subdivision rule $T$ is called \emph{displacement-safe}, if 
\begin{align}\label{displacement-safe}
\dist (Tx_{2i},x_{i}) &\leqslant C\cdot\sup\limits_{\ell} \dist (x_{\ell},x_{\ell+1}), \hspace{0.2cm} i\in \Z 
\end{align}
for some constant $C>0$.

We say that $T$ \emph{converges} for input data $x$ if the sequence $x,Tx,T^2x,\dots$ becomes denser and approaches a continuous limit curve. Formally, we treat convergence in a coordinate chart, linearly interpolating points $T^kx_i$ by a piecewise linear function $f_k$ with $f_k\left(\frac {i}{2^k}\right)=T^kx_i$, and observe convergence of functions $f_k$ for $k \rightarrow \infty$.
It has been shown in various situations that displacement-safe schemes which admit a contractivity factor $\mu <1$ are convergent \cite{dyn3, dyn4,wallner,huening2019}. 
For the reader's convenience we repeat the precise statement which can be found for example in \cite{wallner} and \cite{huening2019}.

\begin{thm} \label{convergence}
Consider a linear, binary, affine invariant subdivision scheme \(S\). Assume that the Riemannian analogue $T$ of \(S\) on a manifold \(M\) is well-defined. Then, \(T\) converges to a continuous limit \(T^{\infty}x\) for all input data \(x\), if it admits a contractivity factor \(\mu <1\) and is displacement-safe.
\end{thm}

\section{Riemannian center of mass on manifolds with positive sectional curvature} \label{part1}
Before we restrict ourselves to the unit sphere, we discuss the difficulties that arise by studying the Riemannian analogue of a linear subdivision scheme on positively-curved manifolds. Let \(M\) be a complete, simply connected Riemannian manifold 
with sectional curvature $K>0$. Denote by $B_r(x)=\lbrace y \in M \mid \dist\left(x,y\right)<r\rbrace$ the geodesic ball of radius $r>0$ around $x\in M$ where $\emph{dist}$ again denotes the Riemannian distance. 

\subsection{Problem setting}
To study the convergence of a Riemannian subdivision rule $T$ as given in (\ref{min:fct}) we have to deal with the question if the function
\begin{align}\label{minimizationfunction2}
f_{\alpha}(x)=\sum_{j=-m}^{m+1} \alpha_j \dist \left(x_j,x\right)^2, 
	\quad\text{with}\quad
	\sum_j \alpha_j =1 
\end{align}
admits a unique minimiser. Here $x_{j}\in M$ are fixed points on the manifold and $\alpha_{j}$ are real coefficients. Later, the points $x_{j}$ correspond to the input data while the coefficients $\alpha_{j}$ belong to the mask of a subdivision scheme. Denote by
\begin{align*}
\alpha_{-}:=\sum_{\alpha_j<0} \vert \alpha_j\vert
\end{align*}
the sum of the absolute values of the negative coefficients.
In contrast to \CH\ manifolds, we cannot hope for general global existence and uniqueness of the Riemannian center of mass. To see this, consider the north pole $x_N$ resp.\ south pole $x_S$ of the sphere and ask for their geodesic midpoint. Clearly, each point on the equator is a suitable choice and thus, a minimiser of $f(x)=\frac 12 \dist\left(x_N,x\right)^2+\frac 12 \dist\left(x_S,x\right)^2$. 
One can show that locally there always exists a unique minimiser while globally there can be infinitely many. 

A number of contributions deals with the question of the effect of the sectional curvature, the distances between the points $x_j$ and the choice of the coefficients on the existence of a unique minimiser, see for example
\cite{karcher, dyer-arxiv}. In \cite{dyer-arxiv} the authors provide explicit bounds on the input data (depending on the curvature and the chosen coefficients) to ensure the existence and uniqueness of a minimiser of (\ref{minimizationfunction2}) on manifolds with positive sectional curvature. 
In our setting, Corollary 9 of \cite{dyer-arxiv} reads as follows.
\begin{lem}[Dyer et al., \cite{dyer-arxiv}] \label{dyer}
Let ${x_j}\in B_r(x)$, $j=-m,\dots,m+1$, for some $x\in M$ and $r>0$. Then, the function $f_{\alpha}$ has a unique minimiser in $B_{r^*}(x)$, if 
\begin{enumerate}[i)]
\item $r<r^*<\min\lbrace \frac{\iota_M}{2},\frac{\pi}{4\sqrt{K}}\rbrace$, with $\iota_M$ denoting the injectivity radius of $M$,
\item $r^* > (1+2\alpha_-)r$,
\item $r^* < \frac{\pi}{4\sqrt{K}}(1+(1+\frac{\pi}{2})\alpha_-)^{-1}$.
\end{enumerate}
\end{lem}


A convergence result for nonlinear subdivision schemes depends on the capability to control the distances of points of the sequence $(T^{k}x_i)_{i\in\mathbb{Z}}$ from each other as well as their distance to the input data. Unfortunately, Lemma \ref{dyer} cannot directly be used to control those distances.


\smallskip
Summarising, on manifolds with positive sectional curvature

\begin{itemize}
\item[i)] we cannot hope for a convergence result which is valid for all input data.
\item[ii)] we obtain a local setting in which the Riemannian analogue of a linear subdivision scheme is well defined, see \cite{dyer-arxiv}.
\item[iii)] we have to find a strategy to estimate distances between consecutive points of the refined data as well as their distance to the input data.
\end{itemize}

\subsection{The Riemannian analogue of a linear subdivision scheme on the unit sphere} \label{Ra_sphere}
From now on, we restrict ourselves to the \emph{unit sphere} $\Sigma^n\subseteq \mathbb{R}^{n+1}$ for $n\geqslant 2$.
In particular, we have $K=1$.
We provide a setting on the unit sphere in which we can define the Riemannian analogue of a linear subdivision scheme.
Choose $x_j \in \Sigma^n$, $j=-m,\dots,m+1$, such that $x_j \in B_r(x)$ for some $r>0$ and $x\in \Sigma^n$.
Since the sectional curvature $K=1$ on the unit sphere and the injectivity radius is $\pi$, Lemma \ref{dyer} says the following: For input data in $B_{r}(x)$ the minimiser of $f_\alpha$ is unique and lies in the ball $B_{r^*}(x)$, if 
\begin{align}
r^* &> \left(1+2\alpha_-\right)r \geqslant r,\label{ii}\\ 
r^* &< \frac{\pi}{4}\left(1+\left(1+\frac{\pi}{2}\right)\alpha_-\right)^{-1}. \label{iii}
\end{align}
In the special case of a scheme with only nonnegative coefficients, i.e., $\alpha_-=0$, these conditions reduce to: If $r<\frac{\pi}{4}$, there exists a radius $r^*$ with $r<r^*<\frac{\pi}{4}$ such that the function $f_\alpha$ has a unique minimiser inside $B_{r^*}(x)$. In our particular setting of subdivision rules, we summarise the results of \cite{dyer-arxiv} as follows. 

\begin{prop} \label{welldefinednesssphere}
Let $T$ be the Riemannian analogue of a linear subdivision scheme $S$ with mask $a$ on the unit sphere $\Sigma^n$. With $\alpha_{-}=\sum_{\alpha_j<0} \vert \alpha_j\vert$, we have the two cases:
\begin{description}
\item[Case $\alpha_-=0$]\hfill \\
$Tx_i$ is well defined if the input data points $x_i$ contributing to the computation of $Tx_i$ lie within a ball of radius $r<\frac{\pi}{4}$. 
\item[Case $\alpha_->0$]\hfill \\
$Tx_i$ is well defined if the input data points $x_i$ contributing to the computation of $Tx_i$ lie within a ball of radius $r$ such that there exists an $r^*>r$ satisfying (\ref{ii}) and (\ref{iii}).
\end{description}
\end{prop}

\section{A strategy to prove convergence of Riemannian subdivision schemes}
\label{part2}



We show a strategy to prove convergence results for the Riemannian analogue $T$ of a linear scheme $S$ on the unit sphere. 
In order to give bounds for $\dist\left(Tx_i,Tx_{i+1}\right)$ and $\dist\left(x_i,Tx_{2i}\right)$ we join the points involved by curves, and estimate the length of those curves. This requires technical details involving a second order Taylor approximation and estimates for the gradient and the Hessian of squared distance functions on the unit sphere. Throughout this part, the cubic Lane-Riesenfeld rule serves as a main example to illustrate our results. We assume that the considered minima are well defined and unique. 



\subsection{The Riemannian distance function on the unit sphere} \label{distsphere}
We use explicit formulas for the gradient and the Hessian of the squared distance function $\dist\left(\cdot,y\right)^2$ computed in \cite[Supplement A]{pennec} as an example of a more general analysis of Hessians of squared distance functions on manifolds. 
We introduce some notation and state results of \cite{pennec} which we later use. 

Let $T_x\Sigma^n=\lbrace w\in \mathbb{R}^{n+1} \mid \langle w, x \rangle =0 \rbrace$ denote the \emph{tangent space} at a point $x\in \Sigma^n$. For two points $x$, $y \in \Sigma^n$, $x\neq -y$, their distance is given by $\dist\left(x,y\right)=\arccos\left(\langle x,y\rangle\right)$. The exponential map at $x\in \Sigma^n$ is given by
\begin{align*}
\exp_x: T_x\Sigma^n &\rightarrow \Sigma^n  \quad
w\mapsto \cos\left(\Vert w \Vert\right)x+\frac{\sin\left(\Vert w \Vert\right)}{\Vert w \Vert}w.
\end{align*} 
The inverse of the exponential map at $x$ is well defined everywhere except for the antipodal point of $x$,
\begin{align} \label{log}
\exp^{-1}_x: \Sigma^n\setminus \lbrace -x\rbrace &\rightarrow T_x\Sigma^n \quad
y\mapsto \frac{\dist\left(x,y\right)}{\sin\left(\dist\left(x,y\right)\right)}(y-\cos\left(\dist\left(x,y\right)\right)x).
\end{align} 
We will always tacitly assume that $\frac{s}{\sin(s)}$ means an analytic function which evaluates to $1$ for $s=0$. For a tangent vector $w \in T_x\Sigma^n$, $\exp_x(w)$ denotes the point on $\Sigma^n$ which is reached by the geodesic starting in $x$ in direction $w$ travelling the length of $\Vert w\Vert$. We can therefore use $\exp_x$ to switch between $\Sigma^n$ and $T_x\Sigma^n$, such that straight lines through the origin in the tangent space are isometrically mapped to geodesics through $x$. 
Let  
\begin{align*} 
g:\Sigma^n &\rightarrow \mathbb{R}
\end{align*}
be a function on the sphere. Then, 
\begin{align*} 
\tilde{g}=g \circ \exp_x:T_x\Sigma^n &\rightarrow \mathbb{R}
\end{align*}
is a representation of $g$ with respect to the coordinate chart $\exp_x^{-1}$. 
Since the first derivative of the exponential map is the identity we have
\begin{align}\label{fct_g2}
\grad\left(g\right)\left(x\right)=\grad\left(\tilde{g}\right)\left(0\right).
\end{align}
As far as the first derivative is concerned, it makes no difference if we consider $g$ or $\tilde{g}$. 
The Hessian $H(\tilde{g})$ can be computed since the function is defined on the linear space $T_x\Sigma^n$. For purposes of this paper, we define the Hessian of $g$ by
\begin{align} \label{fct_g}
H\left(g\right)\left(x\right):=H\left(\tilde{g}\right)\left(0\right).
\end{align} 
For any fixed $y\in M$ the gradient of the squared distance function is given by  
\begin{align}
\grad\left(\dist(\cdot,y\right)^2)(x)=-2\exp^{-1}_x(y).
\end{align}
To simplify notation we introduce the analytic function 
\begin{align*}
\psi(s)=\frac{s}{\tan\left(s\right)}
\end{align*}
with $\psi(0)=0$.
Let $y=\exp_x\left(\rho v\right)$ with $\Vert v\Vert =1$ and $I \in \mathbb{R}^{(n+1)\times (n+1)}$ be the identity matrix. The Hessian of $\dist\left(\cdot,y\right)^2$ (in the sense defined above) has been computed in \cite{pennec} as
\begin{align}
H\left(\dist\left(\cdot,y\right)^2\right)(x)=2\left(vv^T+\psi(\rho)(I-xx^T-vv^T)\right). 
\end{align}
Here $x^T$, $v^T$ denote the \emph{transpose} of column vectors $x$ resp.\ $v$. 
If $x=y$, we have $H\big(\dist\big(\cdot,x\big)^2\big)(x)= 2(I-xx^T)$.

 The eigenvalues of the Hessian are $\lambda_1=0$, $\lambda_2=2$ and $\lambda_3=2\psi\left(\dist\left(x,y\right)\right)$.\\
By linearity we obtain 
\begin{align}
\grad\left(f_\alpha\right)(x)&=-2\sum_{j=-m}^{m+1} \alpha_j \exp^{-1}_x(x_j),\label{sph_gradientf}\\
H\left(f_\alpha\right)(x)=2\sum_{j=-m}^{m+1} \alpha_j &\left(v_jv_j^T+\psi\left(\rho_j\right)\left(I-xx^T-v_jv_j^T\right)\right)\label{sph_hessianf} 
\end{align}
with $x_j=\exp_x\left(\rho_j v_j\right)$, $\Vert v_j\Vert=1$.

\subsection{Taylor approximation of the squared distance function} \label{sec_tay}

The second order Taylor expansion of the squared distance function helps to find an upper bound on the distances between the minimiser of $f_\alpha$ and the input data $x_{j}$. 

Assume that $x^* \in \Sigma^n$ is the minimiser of $f_\alpha$. 
Without loss of generality we choose coordinates such that $x^*=[0,\dots,0,1]^T$. Then, the first $n$ canonical basis vectors span $T_{x^*}\Sigma^n$. Now, we consider the coordinate representation $\tilde{f}_{\alpha}$ and compute its Hessian. 
Due to the particular coordinate system the gradient of $f_\alpha$ consists of the first $n$ entries of the vector given by (\ref{sph_gradientf}). The Hessian is given by the $n\times n$ submatrix of $H\left(f_{\alpha}\right)$ as in (\ref{sph_hessianf}) obtained by deleting the last column and row.  
The second order Taylor approximation of $\tilde{f}_{\alpha}$ is given by
\begin{align*}
T\tilde{f}_\alpha(x)=f_\alpha(x^*)+(x-x^*)^T\grad\left(f_\alpha\right)(x^*)+\frac 12 (x-x^*)^T H\left(f_\alpha\right)(x^*)\left(x-x^*\right).
\end{align*}
Differentiation leads to
\begin{align} 
\grad(T\tilde{f}_\alpha)(x)=\grad(f_\alpha)(x^*)+H(f_\alpha)(x^*)x.
\end{align}
Since we are looking for minimisers of the function $f_\alpha$, the idea is to consider minimisers of $T\tilde{f}_{\alpha}$ instead. If $H\left(f_\alpha\right)(x^*)$ is invertible, then the condition
\begin{align}
\grad\left(f_\alpha\right)(x^*)+H\left(f_\alpha\right)(x^*)x\overset{!}{=}0
\end{align}
is solvable, and $T\tilde{f}_{\alpha}$ assumes a minimum in the point
\begin{align} \label{tay_diff}
x=-H\left(f_\alpha\right)(x^*)^{-1}\grad\left(f_\alpha\right)(x^*).
\end{align}
This $x$ is the unique stationary point of the second order Taylor approximation $T\tilde{f}_{\alpha}$.

After these preparations, we continue with the distance estimates announced above.

\subsection{Variable mask}
We introduce a parameter $t\in [0,1]$ and  vary the coefficients $\alpha_j$ of a linear scheme such that they linearly depend on $t$. The idea is to choose $\alpha_j(t)$ such that at time $t=0$ we exactly know the minimiser of $f_{\alpha(0)}$, call it the \emph{reference point} $\bar{x}$, and at time $t=1$ the minimiser of $f_{\alpha(1)}$ equals $x^*$. We always assume\begin{align} \label{sum_1}
\sum_{j=-m}^{m+1} \alpha_{j}(t)=1 \quad \text{for any} ~t\in [0,1].
\end{align}
We will see that the choice of the reference point is crucial for our approach to work and has to be made individually for each scheme. 
We illustrate the procedure by means of an example.

\begin{ex}[{{\em cubic Lane-Riesenfeld scheme, part I\/}}] \label{Lane_Riesenfeld}
We 
consider the linear cubic Lane-Riesenfeld subdivision rule defined by
\begin{align*} 
(Sx)_{2i}=\frac{1}{8}x_{i-1}+\frac{6}{8}x_{i}+\frac{1}{8}x_{i+1} \quad \text{and} \quad
  (Sx)_{2i+1}=\frac{1}{2}x_{i}+\frac{1}{2}x_{i+1}
\end{align*} 
for $i \in \mathbb{Z}$.
Since the mask has nonnegative coefficients, Proposition \ref{welldefinednesssphere} ensures that the Riemannian version $T$ of $S$ is well defined, if
\begin{align*}
\sup\limits_{\ell} \dist \big(x_{\ell},x_{\ell+1})<\frac{\pi}{4}.
\end{align*}
After contractivity of $T$ is shown (see (\ref{Gleichungxy})), this bound applies also to the iterates $Tx, T^2x,\dots$.

We observe that $Tx_{2i+1}$ is the geodesic midpoint of $x_i$ and $x_{i+1}$. So, its distance to the input data is bounded by $\frac 12 \sup_{\ell}\dist\left(x_\ell,x_{\ell+1}\right)$. The more crucial part is to deal with $\dist\left(Tx_{2i},x_{i}\right)$. Consider 
\begin{align} \label{Lane_Riesenfeld_simple}
 x^*:=\argmin_{x\in \Sigma^n} \left(\frac{1}{8}\dist(x,x_{-1})^2+\frac{6}{8}\dist(x,x_{0})^2+\frac{1}{8}\dist(x,x_{1})^2\right)
 \end{align} 
for $x_{-1}$, $x_{0}$, $x_{1} \in \Sigma^n$.
Without loss of generality we assume that $x_0=[0,\dots, 0,1]^T$. With $\alpha_{-1}=\alpha_{1}=\frac{1}{8}$, $\alpha_0=\frac 34$ as well as $\alpha_2=0$, $x^*$ is the minimiser of $f_\alpha$. 
We choose the coefficient functions as
\begin{align} \label{coeff_func}
\alpha_{-1}(t)=\alpha_1(t)=\frac t8,\quad \alpha_0(t)=1-\frac{t}{4}.
\end{align} 
The minimiser of $f_{\alpha(0)}$ equals $x_0$ while at time $t=1$ the minimiser of $f_{\alpha(1)}$ is exactly the point $x^*$. 
We continue the analysis in Example \ref{Lane-Riesenfeld3}.
\Exende 
\end{ex}

\subsection{Estimating the distance to a minimiser} \label{sec_est}

We introduce the curve $\gamma$ defined as
\begin{align}\label{curve_gamma}
\gamma(t)=\argmin_{x} f_{\alpha(t)}, \quad t \in [0,1].
\end{align}
Since $\gamma$ connects $\bar{x}$ and $x^*$ we have $\dist\left(\bar{x},x^*\right)\leqslant\int_{0}^{1}\Vert \dot{\gamma}(t)\Vert~dt$. The idea is to estimate $\Vert \dot{\gamma}(t)\Vert$ in order to find an upper bound on the distance between $\bar{x}$ and $x^*$. 
If, for example, we choose the reference point $\bar{x}$ to be one of our input data points, this strategy helps us to control the distance between the minimiser $x^*$ of $f_{\alpha(1)}$ and the initial data. 
We start with some assumptions and afterwards show what conclusions we can draw from them.

\begin{assumption} \label{assumption1}
Assume that 
\begin{align*}
\dist\left(x_j,x_{j+1}\right)\leqslant r
\end{align*}
for some constant $r>0$. We choose $r$ such that the minimiser of $f_{\alpha(t)}$ is locally well defined for all $t\in [0,1]$.
\end{assumption}

\begin{assumption} \label{assumption2} Let $r>0$ be as in Assumption \ref{assumption1} and $\gamma$ as in (\ref{curve_gamma}). Assume that 
\begin{align*}
\Vert\dot{\gamma}(0)\Vert \leqslant rC_0
\end{align*}
for some constant $C_0>0$. 
\end{assumption}

\begin{assumption}\label{assumptions} With $r$, $C_0$ as in Assumption \ref{assumption1} resp.\ \ref{assumption2}, assume that the following is true: If  $\Vert \dot{\gamma}(t)\Vert \leqslant rC_0$ for all $t\in [0,1]$, then there exists a constant $C_1<C_0$ such that $\Vert \dot{\gamma}(t) \Vert\leqslant rC_1<rC_0$ for all $t\in[0,1]$.
\end{assumption}

Assumption \ref{assumption1} is necessary for the well-definedness of the Riemannian analogue of a linear scheme. Assumptions \ref{assumption2} and \ref{assumptions} help to estimate the distance between $\bar{x}$ and $x^*$. 


\begin{lem} \label{lem_ana1}
Let $\gamma$ denote the curve which at time $t$ is the minimiser of $f_{\alpha(t)}$. If Assumptions \ref{assumption1}, \ref{assumption2} and \ref{assumptions} are satisfied for $r>0$ and constants $C_0$ and $C_1$, then 
\begin{align*}
\Vert\dot{\gamma}(t)\Vert \leqslant rC_1
\end{align*}
for all $t\in [0,1]$ and $\dist\left(\gamma(0),\gamma(1)\right)\leqslant rC_1$.
\end{lem}

\begin{proof}
Let $t^*=\sup\lbrace t \in [0,1] \mid \Vert \dot{\gamma}(t)\Vert \leqslant rC_1\rbrace$. 
Then,
\begin{align*}
\Vert \dot{\gamma}(t^*)\Vert = \lim_{t<t^*}\Vert \dot{\gamma}(t)\Vert\leqslant rC_1.
\end{align*}
Assume that $t^*<1$. Since $\Vert \dot{\gamma}(t)\Vert$ is continuous, there exists an interval $J=(t^*-\epsilon,t^*+\epsilon)$, $\epsilon>0$, with $\Vert \dot{\gamma}(\tilde{t})\Vert\leqslant rC_1$ for any $\tilde{t} \in J$. But this is a contradiction to $t^*$ being maximal.
\end{proof}

We continue with the analysis of a model subdivision rule.

\begin{ex}[{{\em cubic Lane-Riesenfeld scheme, part II\/}}] \label{Lane-Riesenfeld3}
We have 
\begin{align*}
H\left(f_{\alpha(0)}\right)(x_0)&=2\alpha_{0}(0)\left(v_0v_0^T+\left(I-x_0x_0^T-v_0v_0^T\right)\right)=2I.
\end{align*}
Recall that the second equality is based on the assumption $x_0=[0,\dots,0,1]^T$. In particular, the inverse  $H\left(f_{\alpha(0)}\right)(x_0)^{-1}=\frac 12 I$ exists.
By (\ref{sph_gradientf}) we have
\begin{align*}
&\frac{d}{dt}\bigg|_{t=0}\grad\left(f_{\alpha(t)}\right)(x_0)=-2\frac{d}{dt}\bigg|_{t=0}\sum_{j=-1}^{1} \alpha_j(t) \exp^{-1}_{x_0}(x_j)\\
&=-2\left(\frac 18 \exp^{-1}_{x_0}(x_{-1})+\frac 18 \exp^{-1}_{x_0}(x_{1})\right)=-\frac 14  \left(\exp^{-1}_{x_0}(x_{-1})+\exp^{-1}_{x_0}(x_{1})\right),
\end{align*}
using the geometric fact $\exp^{-1}_{x_0}(x_{0})=0$.
We conclude that
\begin{align*}
\dot{\gamma}(0)=\frac 18 \left(\exp^{-1}_{x_0}(x_{-1})+\exp^{-1}_{x_0}(x_{1})\right).
\end{align*}
Assuming that $\dist\left(x_{j},x_{j+1}\right)\leqslant r$ for some $0<r<\frac{\pi}{4}$ and $j=-1,0$, the above shows that
\begin{align*}
\Vert\dot{\gamma}(0)\Vert \leqslant \frac 14 r.
\end{align*}
This is a first piece of information needed to establish constants $C_0$, $C_1$, and eventually prove convergence of the cubic Lane-Riesenfeld scheme.
\Exende
\end{ex}

We are now estimating the norm of $\dot{\gamma}\left(t\right)$ for some fixed $t \in [0,1]$. The computations are done in the tangent space $T_{\gamma(t)}\Sigma^n$ where for simplicity we always assume that $\gamma(t)$ is the north pole $[0,\dots,0,1]^T$. Of course, the coordinates of the $x_j$'s change for different $t$, but since we only consider distances which are independent of the chosen coordinate system, we do not indicate the coordinate change in the notation.
The Hessian $H\left(\dist(\cdot,y)^2\right)(\gamma(t))$ of the squared distance function has the eigenvalues $\lambda_1=2$ and $\lambda_2=2\psi\left(\dist\left(\gamma(t),x_j\right)\right)$, see Section \ref{distsphere}. In particular, we have $\dist\left(\gamma(t),x_j\right)<\frac{\pi}{2}$, $j=-m,\dots,m+1$, since the radius $r^*$ of the ball containing the input data and the minimiser $\gamma(t)$ is smaller than $\frac{\pi}{4}$, see Section \ref{Ra_sphere}. That is why $\lambda_2\leqslant\lambda_1$ as well as $0<\lambda_2\leqslant 2$. So, we know that the inverse of the Hessian exists.
The location $\gamma(t)$ of the minimiser is implicitly defined by $\grad(f_{\alpha(t)})=0$, so the derivative $\dot{\gamma}(t)$ is determined by the derivatives $Hf_{\alpha}(t)$, $\frac{d}{dt}\grad(f_{\alpha(t)})$. Therefore, for purpose of computing $\dot{\gamma}(t)$, we can replace $f_{\alpha(t)}$ resp. $\tilde{f}_{\alpha(t)}$ by the 2nd order Taylor expansion $T\tilde{f}_{\alpha(t)}$. From (\ref{tay_diff}) we get

\begin{align} \label{x_appro}
\dot{\gamma}(t)&=-\frac{d}{ds}\bigg|_{s=t}\big(\left(H\left(f_{\alpha(s)}\right)\left(\gamma(t)\right)\right)^{-1}\grad\left(f_{\alpha(s)}\right)\left(\gamma(t)\right)\big)\\&=-\left(H\left(f_{\alpha(t)}\right)\left(\gamma(t)\right)\right)^{-1}\frac{d}{ds}\bigg|_{s=t}\grad\left(f_{\alpha(s)}\right)\left(\gamma(t)\right), \notag
\end{align}
where we have used $\grad\left(f_{\alpha(t)}\right)\left(\gamma(t)\right)=0$.

The following two lemmas estimate the spectral norm of the inverse of the Hessian and the derivative of the gradient in order to find an upper bound on $\Vert \dot{\gamma}(t)\Vert$.

 

\begin{lem} \label{Hess_lem}
Assume that $\dist\left(x_j,x_{j+1}\right)\leqslant r$ for some $r>0$ and that $\Vert \dot{\gamma}(t) \Vert \leqslant C_0r$ for $C_0>0$ and all $t\in [0,1]$. Let $\ell_j$ be constants such that $\dist\left(x_j,\bar{x}\right)\leqslant \ell_j$. Then,
\begin{align*}
\Vert \left(H\left(f_{\alpha(t)}\right)\left(\gamma(t)\right)\right)^{-1} \Vert\leqslant \frac{1}{\vert 2-L(t)\vert}
\end{align*}
with $\displaystyle L(t)=\sum_{j=-m}^{m+1} \vert\alpha_{j} (t)\vert\left(2-2\psi\left(C_0rt+\ell_j\right) \right)$ for all $t \in [0,1]$.
\end{lem}

\begin{proof}
We have
\begin{align*}
\Vert H\left(f_{\alpha(t)}\right)\left(\gamma(t)\right) \Vert &=\Big\Vert \sum_{j=-m}^{m+1}\alpha_{j}(t)H\big(\dist\left(x_j,\cdot\right)^2\big)\left(\gamma(t)\right)\Big\Vert \leqslant 2\sum_{j=-m}^{m+1}\vert\alpha_{j}(t)\vert,
\end{align*}
since the maximal eigenvalue of the Hessian of the squared distance function is $\lambda_1=2$. In particular, the norm of any eigenvalue of the Hessian is bounded from above by $2\sum_{j=-m}^{m+1}\vert\alpha_{j}(t)\vert$.
Furthermore, we see that
\begin{align}\label{estHess1}
\Vert 2I-H\left(f_{\alpha(t)}\right)\left(\gamma(t)\right) \Vert &=\Big\Vert \sum_{j=-m}^{m+1}\alpha_{j}(t)\big(2I-H\big(\dist\left(x_j,\cdot\right)^2\big)\left(\gamma(t)\right)\big)\Big\Vert \notag \\
&\leqslant \sum_{j=-m}^{m+1}\vert\alpha_{j}(t)\vert\Big \Vert 2I-H\big(\dist\left(x_j,\cdot\right)^2\big)\left(\gamma(t)\right)\Big\Vert.
\end{align}
Denote by $\lambda_{2,j}(t)$ the smaller eigenvalue of  $H\big(\dist\left(x_j,\gamma(t)\right)^2\big)$. In fact,
\begin{align*}
\lambda_{2,j}(t)=2\psi\left(\dist\left(\gamma(t),x_j\right)\right)<2.
\end{align*}
Since $\Vert \dot{\gamma}(t) \Vert \leqslant C_0r$ and the fact that $\psi$ is positive and monotonically decreasing in $[0,\frac{\pi}{2}]$ we deduce that
\begin{align*}
\lambda_{2,j}(t)\geqslant 2\psi\left(C_0rt+\ell_j\right).
\end{align*}
By (\ref{estHess1}) we therefore obtain 
\begin{align*}
\Vert 2I-H\left(f_{\alpha(t)}\right)\left(\gamma(t)\right) \Vert &\leqslant L(t)
\end{align*}
and the minimal eigenvalue of $ H\left(f_{\alpha(t)}\right)\left(\gamma(t)\right)$ is bounded from below by $\vert 2-L(t) \vert$. This implies the statement of the lemma.
\end{proof}

\begin{lem} \label{cons2}
Assume that $\dist\left(x_j,x_{j+1}\right)\leqslant r$ for some $r>0$ and that $\Vert \dot{\gamma}(t) \Vert \leqslant C_0r$ for $C_0>0$ and all $t\in [0,1]$.
Let $\ell_j$ be constants such that $\dist\left(x_j,\bar{x}\right)\leqslant \ell_j$. Then,
\begin{align*}
\Big\Vert \frac{d}{ds}\bigg|_{s=t}\grad\left(f_{\alpha(s)}\right)\left(\gamma(t)\right)\Big\Vert&\leqslant 2\sum_{j=-m}^{m+1} \vert\dot{\alpha}_j(t)\vert \left(rC_0t+\ell_j\right)
\end{align*}
for all $ t \in [0,1]$. 
\end{lem}

\begin{proof}
Fix some $t \in [0,1]$. Since $\dist\left(\gamma(t),x_j\right)=\Vert \exp^{-1}_{\gamma(t)}(x_j)\Vert$ and, by assumption, $\Vert \dot{\gamma}(t) \Vert \leqslant C_0r$ we deduce that
\begin{align*}
\Vert \exp^{-1}_{\gamma(t)}(x_j)\Vert&\leqslant \Vert \exp^{-1}_{\gamma(t)}(\bar{x})\Vert+\Vert \exp^{-1}_{\bar{x}}(x_j)\Vert \leqslant rC_0t+\ell_j.
\end{align*}
So, (\ref{sph_gradientf}) implies that  
\begin{align*}
\Big\Vert \frac{d}{ds}\bigg|_{s=t}\grad\left(f_{\alpha(s)}\right)\left(\gamma(t)\right)\Big\Vert &\leqslant 2\sum_{j=-m}^{m+1} \bigg|\frac{d}{ds}\bigg|_{s=t}\alpha_j(s) \bigg| \Vert\exp^{-1}_{\gamma(t)}(x_j)\Vert \\
&\leqslant 2\sum_{j=-m}^{m+1}  \vert\dot{\alpha}_j(t)\vert \left(rC_0t+\ell_j\right).
\end{align*}
\end{proof}

We summarise the results of the previous two lemmas in 
\begin{prop} \label{est_norm}
Assume that $\dist\left(x_j,x_{j+1}\right)\leqslant r$ for some $r>0$ and that $\Vert \dot{\gamma}(t) \Vert \leqslant C_0r$ for $C_0>0$ and all $t\in [0,1]$. Let $\ell_j$ be constants such that $\dist\left(x_j,\bar{x}\right)\leqslant \ell_j$. Then,
\begin{align} \label{appro_x(t)}
\Vert \dot{\gamma}(t) \Vert \leqslant \frac{2}{\vert2-L(t)\vert}\sum_{j=-m}^{m+1} \vert\dot{\alpha}_j(t)\vert \left(rC_0t+\ell_j\right)
\end{align}
with $\displaystyle L(t)=\sum_{j=-m}^{m+1} \vert\alpha_{j} (t)\vert\left(2-2\psi\left(C_0rt+\ell_j\right)\right)$ for all $t \in [0,1]$.
\end{prop}

\begin{proof}
By (\ref{x_appro}) we have
\begin{align*}
\Vert \dot{\gamma}(t) \Vert &\leqslant \Big\| \left(H\left(f_{\alpha(t)}\right)\left(\gamma(t)\right)\right)^{-1}\frac{d}{ds}\bigg|_{s=t}\grad\left(f_{\alpha(s)}\right)\left(\gamma(t)\right) \Big\|\\
&\leqslant \Big\| \left(H\left(f_{\alpha(t)}\right)\left(\gamma(t)\right)\right)^{-1}\Big\| \Big\|\frac{d}{ds}\bigg|_{s=t}\grad\left(f_{\alpha(s)}\right)\left(\gamma(t)\right) \Big\|
\end{align*} 
for all $t\in [0,1]$. The statement then follows from Lemma \ref{Hess_lem} and Lemma \ref{cons2}.
\end{proof}

We illustrate the results of Proposition \ref{est_norm} by means of our main example.

\begin{ex}[{{\em cubic Lane-Riesenfeld scheme, part III\/}}] \label{Lane-Riesenfeld4}
Lemma \ref{Hess_lem} shows that
\begin{align*}
L(t)=\frac{t}{4}\left(2-2\psi\left(C_0rt+r\right)\right)+\left(1-\frac t4\right)\left(2-2\psi\left(C_0rt\right)\right)
\end{align*}
for all $t \in [0,1]$ with $\ell_0=0$ and $\ell_{-1}=\ell_1=r$.
Since this function is strictly increasing in the interval $[0,1]$ we have 
\begin{align*}
L(t)\leqslant \frac{1}{4}\left(2-2\psi\left(C_0r+r\right)\right)+\frac{3}{4}\left(2-2\psi\left(C_0r\right)\right)
\end{align*}
for all $t \in [0,1]$.
We conclude that
\begin{align*}
\Vert \left(H\left(f_{\alpha(t)}\right)\left(\gamma(t)\right)\right)^{-1} \Vert&\leqslant \frac{1}{2-\left(2-\frac 12 \psi\left(C_0r+r\right)-\frac 32 \psi\left(C_0r\right)\right)}=\frac{1}{\frac 12 \psi\left(C_0r+r\right)+\frac 32 \psi\left(C_0r\right)}.
\end{align*}
This bound only depends on $C_0$ and $r$.
This estimate is needed for the verification of Assumption \ref{assumptions} of our method.

Remember that we have chosen $\bar{x}=x_0$ as well as
\begin{align} \label{coeff_func}
\alpha_{-1}(t)=\alpha_1(t)=\frac t8 \quad \text{and} \quad \alpha_0(t)=1-\frac{t}{4}, \quad t \in [0,1].
\end{align}
Assume that $r<\frac{\pi}{4}$ is such that $\dist\left(x_{j},x_0\right)\leqslant r$, for $j=-1,1$. Since $\sum_{j=-1}^1 \vert\dot{\alpha_j}(t)\vert=\frac {1}{2}$ Equation (\ref{appro_x(t)}) reads
\begin{align}\label{eq_norm}
\Vert \dot{\gamma}(t) \Vert &\leqslant  \frac{2}{\frac 12 \psi\left(C_0r+r\right)+\frac 32 \psi\left(C_0r\right)}\left(2\dot{\alpha}_1(t)r+\frac 12 rC_0t\right) \notag \\
&=\frac{2}{\frac 12 \psi\left(C_0r+r\right)+\frac 32 \psi\left(C_0r\right)}\left(\frac 14 r+\frac 12 rC_0t\right)
\end{align}
for all $t\in[0,1]$. In order to make this bound a proof of Assumption \ref{assumptions}, we must do some experimenting. For the sake of demonstration, choose $r_0=\frac 14$ and $C_0=0.53$. Then $\Vert \dot{\gamma}(0) \Vert <rC_0$ for any $0< r\leqslant \frac 14$, see Example \ref{Lane-Riesenfeld3}. In particular, Assumption \ref{assumption2} is satisfied for all $0<r\leqslant r_0$. Next, compute 
\begin{align*}
\frac 12 \psi\left(C_0r_0+r_0\right)+\frac 32 \psi\left(C_0r_0\right) \approx 1.97.
\end{align*}
Since $\psi\left(s\right)$ is positive and monotonically decreasing in $[0,\frac{\pi}{2}]$, we conclude 
\begin{align*}
\frac{2}{\frac 12 \psi\left(C_0r+r\right)+\frac 32 \psi\left(C_0r\right)}\left(\frac{r}{4}+\frac{rC_0}{2}\right) &\leqslant\frac{2}{\frac 12 \psi\left(C_0r_0+r_0\right)+\frac 32 \psi\left(C_0r_0\right)}\left(\frac{r}{4}+\frac{rC_0}{2}\right) \\
&\approx 0.52r
\end{align*}
for any $0< r\leqslant r_0$. By (\ref{eq_norm}) this expresses Assumption \ref{assumptions} with $C_1\approx 0.52$, $C_0=0.53$ and all $0< r\leqslant r_0$. Lemma \ref{lem_ana1} then says
\begin{align*}
\dist\left(x_0, x^*\right)\leqslant 0.39r
\end{align*}
for any $0< r\leqslant r_0$. The computations above have been performed such that this inequality, featuring a $2$-digit rounded number, is correct.

\begin{figure}
\centering\unitlength0.0098\textwidth
\begin{picture}(60,33)\put(-6,0){
\put(20,0){\includegraphics[width=30\unitlength]{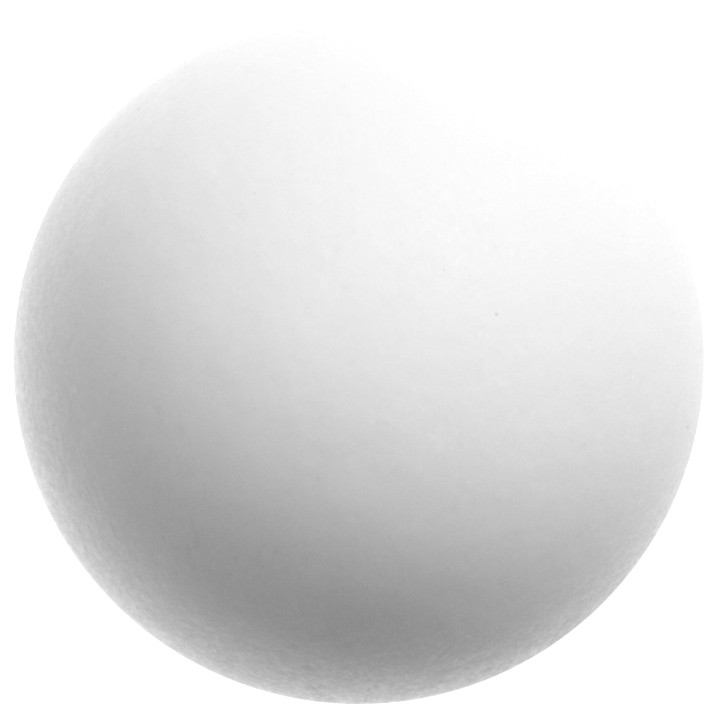}}
\put(20,0){\includegraphics[width=30\unitlength]{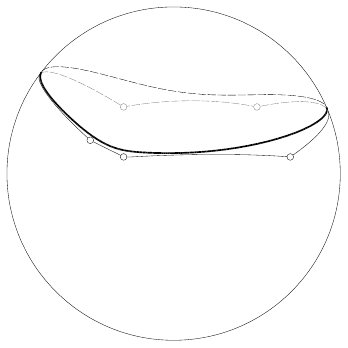}}
}\end{picture}
\caption{Initial polygon and limit curve of the cubic Lane- Riesenfeld scheme on the unit sphere.}
\label{lr_plot}
\end{figure}

We are now in a position to analyse the convergence of the Riemannian analogue $T$ of the linear cubic Lane-Riesenfeld scheme. If we choose input data $(x_i)_{i\in \mathbb{Z}}$ with 
\begin{align*}
\sup\limits_{\ell} \dist\left(x_\ell,x_{\ell+1}\right)<r_0,
\end{align*} 
then our previous computations together with the fact that the points $Tx_{2i+1}$ are the geodesic midpoints of $x_i$, $x_{i+1}$ lead to
\begin{align} \label{Gleichungxy}
\dist\left(Tx_i,Tx_{i+1}\right)\leqslant \frac{r_0}{2}+0.39r_0=0.89r_0, \quad \dist\left(Tx_{2i},x_i\right)\leqslant   0.39 r_0
\end{align}
for all $i \in \mathbb{Z}$.

These inequalities show contractivity and the displacement-safe condition.
Further numerical experiments show that one can even choose $r_0=0.6$ and $C_0=0.69$. In that case, $C_1\approx 0.68<C_0$. See Figure \ref{lr_plot} for an example of the cubic Lane-Riesenfeld scheme on the unit sphere.
\Exende
\end{ex}

Summing up, we formulate
\begin{cor}
Let $(x_i)_{i\in \mathbb{Z}}$ be a sequence of points on the unit sphere. If 
\begin{align*}
\sup\limits_{\ell} \dist(x_\ell,x_{\ell+1})<0.6 \approx 0.19\pi,
\end{align*}
then the Riemannian analogue of the linear cubic Lane-Riesenfeld scheme converges to a continuous, even $C^2$, limit function on the unit sphere.
\end{cor}

The $C^2$ statement follows from \cite{grohs}.


\subsection{Example: $4$-point scheme}
We consider the well-known $4$-point scheme defined by
\begin{align}
	(Sx)_{2i}=x_i 
	\hspace{0.5cm} \text{and} \hspace{0.5cm} 
	(Sx)_{2i+1}=-\omega x_{i-1}
		+\Big(\frac1{2} +\omega\Big) x_i
		+\Big(\frac1{2}+\omega\Big) x_{i+1}-\omega x_{i+2},
	\end{align}
for some parameter \(\omega\), $i\in \mathbb{Z}$, see \cite{dyn1}. 
We analyse the $4$-point scheme for $\omega=\frac{1}{16}$, see Figure \ref{vierpunkt_plot} for an example.
\begin{figure}
\centering\unitlength0.0098\textwidth
\begin{picture}(60,33)\put(-6,0){
\put(20,0){\includegraphics[width=30\unitlength]{sphere1.jpg}}
\put(20,0){\includegraphics[width=30\unitlength]{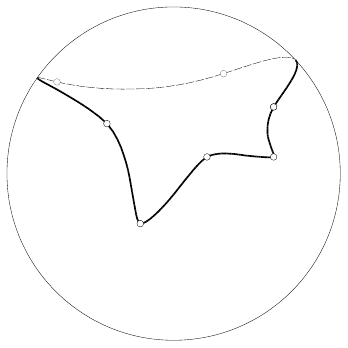}}
}\end{picture}
\caption{Limit curve of the $4$-point scheme applied to input data on the unit sphere.}
\label{vierpunkt_plot}
\end{figure}
First, we focus on 
\begin{align} \label{fourpoint_simple}
 x^*:=\argmin_{x\in \Sigma^n} \Big(&-\frac{1}{16}\dist\left(x,x_{-1}\right)^2+\frac{9}{16}\dist\left(x,x_{0}\right)^2\\
 &+\frac{9}{16}\dist\left(x,x_{1}\right)^2-\frac{1}{16}\dist\left(x,x_{2}\right)^2\Big) \notag
 \end{align} 
for some $x_j\in \Sigma^n$. Let $\dist\left(x_j,x_{j+1}\right)<r_0$ for some $r_0>0$. 
In particular, we have $\alpha_-=\frac 18$. Thus, 
Conditions (\ref{ii}), (\ref{iii}) and Proposition \ref{welldefinednesssphere} imply that  
if $r_0< \frac{0.59}{1.25}\cdot\frac{2}{3}\approx 0.31$, our input data lies inside a ball of small enough radius such that the minimiser $x^*$ is well defined.

So, let $r_0=0.31$ and denote the geodesic midpoint of $x_0$ and $x_1$ by $x_m$. 
Observe that due to our restrictions on the input data $x_m$ is well defined. Choose coefficient functions
\begin{align} \label{coeff_func_fourpoint}
\alpha_{-1}(t)=\alpha_2(t)=-\frac {t} {16} \quad \text{and} \quad \alpha_0(t)=\alpha_1(t)=\frac 12 +\frac{t}{16}
\end{align} 
for $t \in [0,1]$. Let $\gamma(t)$ be the minimiser of $f_{\alpha(t)}$. Then, the minimum of $f_{\alpha(0)}$ equals $x_m$, while the minimum of $f_{\alpha(1)}$ is the point $x^*$.
We see that
\begin{align}\label{GL_XX}
\Big\Vert\frac{d}{dt}\bigg|_{t=0}&\grad\left(f_{\alpha(t)}\right)\left(x_m\right)\Big\Vert\leqslant 2\Big\Vert\frac{d}{dt}\bigg|_{t=0}\sum_{j=-1}^{2} \alpha_j(t) \exp^{-1}_{x_m}(x_j)\Big\Vert\\
&\leqslant 2 \left( \frac {1}{16} \exp^{-1}_{x_m}(x_{-1})+ \frac 1 {16} \exp^{-1}_{x_m}(x_{0})+ \frac 1 {16} \exp^{-1}_{x_m}(x_{1})+\frac {1}{16} \exp^{-1}_{x_m}(x_{2})\right) \notag\\ 
&\leqslant 2 \left( \frac {2}{16} \cdot\frac{3r}{2}+\frac 2 {16} \cdot\frac{r}{2}\right)=2\cdot \frac{1}{4}r \notag
\end{align}
for any $0<r\leqslant r_0$.
Moreover, we deduce that
\begin{align*}
L(t)=\frac{2t}{16}\left(2-2\psi\left(C_0rt+\frac{3}{2}r\right)\right)+2\left(\frac 12 +\frac{t}{16}\right)\left(2-2\psi\left(C_0rt+\frac{1}{2}r\right)\right)
\end{align*}
for a constant $C_0$, $L(t)$ as in Lemma \ref{Hess_lem} and all $t \in [0,1]$. In particular, 
\begin{align*}
L(0)=2-\frac{r}{\tan\left(\frac{1}{2}r\right)}.
\end{align*}
Considered as a function in $r$, $L(0)$ is positive and monotonically increasing for $0<r\leqslant r_0$. Thus, 
$L(0)\leqslant 2-\frac{r_0}{\tan\left(\frac{1}{2}r_0\right)} \approx 0.02$ and $\frac{2}{2-0.02}\cdot\frac{1}{4}r= \frac{50}{198} r$.
Lemma \ref{Hess_lem}, together with our previous computations, yields
\begin{align*}
\Vert \dot{\gamma}(0) \Vert \leqslant  \frac{50}{198} r
\end{align*}
for all $0<r\leqslant r_0$ and Assumption \ref{assumption2} is satisfied for any constant $C_0 > \frac{50}{198}$.
We assume that $\Vert \dot{\gamma}(t) \Vert < C_0r$ for some $0 < r \leqslant r_0$ and all $t\in [0,1]$. Again by monotonicity $L(t)\leqslant L(1)$ and by Proposition \ref{est_norm} we therefore have
\begin{align*}
\Vert \dot{\gamma}(t) \Vert &\leqslant \frac{2}{\vert 2-L(1)\vert}\left( \frac{2}{16}\left(rC_0t+\frac{3}{2}r\right)+\frac{2}{16}\left(rC_0t+\frac{r}{2}\right)\right)\\
&\leqslant  \frac{2}{\vert 2-L(1)\vert} \left( \frac{1}{4}rC_0+\frac 14 r\right)
\end{align*}
for all $t\in [0,1]$. This implies that
\begin{align*}
\dist\left(x_m,x^*\right)&\leqslant 
\frac{2}{\vert 2-L(1)\vert} \left(\frac{1}{8}rC_0+\frac 14 r\right) \quad \text{and} \quad \dist\left(x_0,x^*\right)\leqslant \dist\left(x_m,x^*\right)+\frac{r}{2}
\end{align*}
for any $0<r\leqslant r_0$. Now, we ask $\dist\left(x_m,x^*\right)+\frac{r}{2}<r$ to obtain contractivity as well as a displacement-safe condition. Thus, we are looking for a constant $C_0> 0.26$ together with suitable choices for $r$ such that 
\begin{align*}
 \frac{2}{\vert 2-L(1)\vert}\left( \frac 14 C_0+\frac 14 \right)<C_0
\quad \text{and} \quad
\frac{2}{\vert 2-L(1)\vert}\left( \frac 18 C_0+\frac 14 \right)<\frac 12.
\end{align*}
Numerical computations show that if $C_0=0.45$, both inequalities are satisfied for any $0< r \leqslant r_0=0.31$.
Thus, the Riemannian analogue $T$ of the linear $4$-point scheme is displacement-safe and the maximal distance of consecutive points $T^kx_{i}$, $T^kx_{i+1}$ strictly decreases if $k$ goes to infinity. We have shown 

\begin{cor}
Let $(x_i)_{i\in \mathbb{Z}}$ be a sequence of points on the unit sphere. If 
\begin{align*}
\sup\limits_{\ell} \dist\left(x_\ell,x_{\ell+1}\right)<0.31\approx 0.10\pi,
\end{align*}
then the Riemannian analogue of the linear $4$-point scheme with parameter $\omega=\frac{1}{16}$ converges to a continuous, even $C^1$, limit function on the unit sphere.
\end{cor}

The $C^1$ property follows from \cite{grohs5}.

While we have chosen an input data point as reference point in the first example, we have now seen that the choice of a geodesic midpoint of two initial data points yields a convergence result. 
The next example shows that our strategy works for non-interpolatory schemes with negative coefficients, too.

\subsection{Example: A scheme with negative coefficients}
We consider the linear scheme 
\begin{align} \label{neg_mask_exp}
	(Sx)_{2i}
	&=-\frac1{32}x_{i-1}
	+\frac{21}{32}x_i
	+\frac{13}{32}x_{i+1}
	-\frac1{32}x_{i+2}, 
	\\
	(Sx)_{2i+1}
	&=-\frac1{32}x_{i-1}
	+\frac{13}{32}x_i
	+\frac{21}{32}x_{i+1}
	-\frac1{32}x_{i+2}, \notag
\end{align} 
$i \in \mathbb{Z}$.
Because of the symmetry of the two refinement rules it is sufficient to analyse
\begin{align*} 
 x^*:=\argmin_{x\in \Sigma^n} \Big(&-\frac{1}{32}\dist\left(x,x_{-1}\right)^2+\frac{21}{32}\dist\left(x,x_{0}\right)^2 \\
 &+\frac{13}{32}\dist\left(x,x_{1}\right)^2-\frac{1}{32}\dist\left(x,x_{2}\right)^2\Big) \notag
 \end{align*} 
with $x_j\in \Sigma^n$. Let $\dist\left(x_j,x_{j+1}\right)<r_0$ for some $r_0>0$. 

Since $\alpha_-=\frac {1} {16}$, Conditions (\ref{ii}), (\ref{iii}) and Proposition \ref{welldefinednesssphere} imply that
if $r_0< \frac{0.68}{1.125}\cdot\frac{2}{3}\approx 0.4$, our input data lies inside a ball of small enough radius such that the minimiser $x^*$ is well defined.
So, let $r_0=0.4$. We choose as reference point $\bar{x} \in \Sigma^n$ to be the weighted geodesic average of $x_0$ and $x_1$ with weights $\beta_0=0.65$ and $\beta_1=0.35$. These numbers were found by some experimenting.
 Define coefficient functions
\begin{align*} 
\alpha_{-1}(t)=\alpha_2(t)=-\frac {t}{32}, \quad \alpha_0(t)=\frac{65}{100}+\frac{t}{160} \quad \text{and} \quad \alpha_1(t)=\frac{35}{100} +\frac {9}{160}t
\end{align*} 
for $t \in [0,1]$ and let $\gamma$ denote the curve connecting the minimisers of $f_{\alpha(t)}$.
Then, the minimiser of $f_{\alpha(0)}$ is $\bar{x}$, while the minimiser of $f_{\alpha(1)}$ is $x^*$. Analogous to (\ref{GL_XX}) we get
\begin{align*}
\Big\Vert\frac{d}{dt}\bigg|_{t=0}\grad\left(f_{\alpha(t)}\right)(\bar{x})\Big\Vert\leqslant 2 \left( \frac {1}{32} \frac{135}{100}r+\frac 1 {160}\frac{35}{100}r +\frac{9}{160}\frac{65}{100}r+\frac 1{32} \frac{165}{100}r\right)= 2\cdot \frac{53}{400}r
\end{align*}
for any $0<r\leqslant r_0$.
\begin{figure}
\centering\unitlength0.0098\textwidth
\begin{picture}(60,33)\put(-6,0){
\put(20,0){\includegraphics[width=30\unitlength]{sphere1.jpg}}
\put(20,0){\includegraphics[width=30\unitlength]{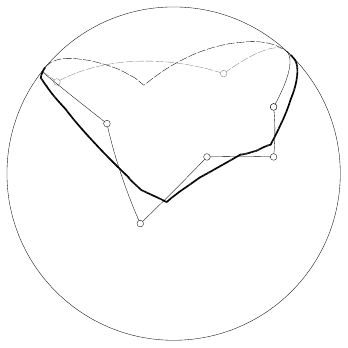}}
}\end{picture}
\caption{Initial polygon and limit curve of the nonlinear analogue of the linear subdivision rule (\ref{neg_mask_exp}) on the unit sphere.}
\label{chvierpunkt_plot}
\end{figure}
Moreover, we deduce that
\begin{align*}
L(t)&=\frac{t}{32}\left(2-2\psi\left(C_0rt+1.35r\right)\right)+\left(0.65 +\frac{t}{160}\right)\left(2-2\psi\left(C_0rt+0.35r\right)\right)\\
&~~~+\left(0.35 +\frac{9}{160}t\right)\left(2-2\psi\left(C_0rt+0.65r\right)\right)+\frac{t}{32}\left(2-2\psi\left(C_0rt+1.65r\right)\right)
\end{align*}
for some constant $C_0$, $L(t)$ as in Lemma \ref{Hess_lem} and all $t \in [0,1]$. Considered as a function in $r$, $L(0)$ is positive an monotonically increasing for $0< r\leqslant r_0$. Thus, $L(0)\leqslant 0.02$ and $\frac{2}{2-0.02}\cdot \frac{53}{400}r \approx 0.134 r$.
Lemma \ref{Hess_lem} and our previous computations show that
\begin{align*}
\Vert \dot{\gamma}(0) \Vert \leqslant 0.14 r
\end{align*}
for all $0<r\leqslant r_0$. So, Assumption \ref{assumption2} is satisfied for any constant $C_0 > 0.14$.
We assume that $\Vert \dot{\gamma}(t) \Vert < C_0r$ for some $0 < r \leqslant r_0$ and all $t\in [0,1]$. Again by monotonicity $L(t)\leqslant L(1)$ and by Proposition \ref{est_norm} we therefore deduce that
\begin{align*}
\Vert \dot{\gamma}(t) \Vert &\leqslant \frac{2}{\vert 2-L(1)\vert}\Big( \frac{1}{32}\left(rC_0t+1.35r\right)+\frac{1}{160}\left(rC_0t+0.35r\right)\\
&~~~+\frac{9}{160}\left(rC_0t+0.65r\right)+\frac{1}{32}\left(rC_0t+1.65r\right)\Big)\\
&=  \frac{2}{\vert 2-L(1)\vert} \left( \frac{1}{8}rC_0t+\frac{53}{400}r\right)\leqslant  \frac{2}{\vert 2-L(1)\vert} \left( \frac{1}{8}rC_0+\frac{53}{400} r\right)
\end{align*}
for all $t\in [0,1]$. This implies that
\begin{align*}
\dist\left(\bar{x},x^*\right)&\leqslant 
\frac{2}{\vert 2-L(1)\vert} \left(\frac{1}{16}rC_0+\frac{53}{400} r\right)
\quad \text{and} \quad
\dist\left(x_0,x^*\right)\leqslant \dist\left(\bar{x},Tx\right)+0.35r
\end{align*}
for any $0<r\leqslant r_0$. Now, we ask $\dist\left(\bar{x},x^*\right)+0.35r<\frac{r}{2}$. 
Thus, we are looking for a constant $C_0>0.14$ together with suitable choices for $r$ such that 
\begin{align*}
 \frac{2}{\vert 2-L(1)\vert}\left( \frac 18 C_0+\frac{53}{400} \right)<C_0
 \quad \text{and} \quad
\frac{2}{\vert 2-L(1)\vert}\left( \frac {1}{16} C_0+\frac{53}{400} \right)<\frac 12-0.35=0.15.
\end{align*}

If $C_0=0.16$, numerical computations show that both inequalities are satisfied for any $0< r \leqslant r_0=0.4$. Thus, the Riemannian analogue $T$ (see Figure \ref{chvierpunkt_plot}) admits a contractivity factor less than 1 and is displacement-safe.
We have shown 

\begin{cor}
Let $(x_i)_{i\in \mathbb{Z}}$ be a sequence of points on the unit sphere. If 
\begin{align*}
\sup\limits_{\ell\in \mathbb{Z}} \dist\left(x_\ell,x_{\ell+1}\right)<0.4\approx 0.13 \pi,
\end{align*}
then the Riemannian analogue of the linear subdivision scheme defined in (\ref{neg_mask_exp}) converges to a continuous limit function on the unit sphere.
\end{cor}


\subsection*{Acknowledgements}
The authors acknowledge the support of the Austrian Science Fund (FWF) under grant no.\ W1230.
The results of this paper are part of the PhD thesis of the first author.
\bibliographystyle{amsplain}
\bibliography{literatur}{}
	
\end{document}